\documentclass{amsart}[12pt]
\input xy
\xyoption{all}
\newcommand{\Q}{{\mathbb Q}}
\newtheorem{theorem}{Theorem}
\newtheorem{proposition}[theorem]{Proposition}
\newtheorem{corollary}[theorem]{Corollary}
\newtheorem{lemma}[theorem]{Lemma}
\theoremstyle{definition}
\newtheorem{definition}[theorem]{Definition}
\newtheorem{example}[theorem]{Example}
\newtheorem{convention}[theorem]{Convention}
\newtheorem{problem}[theorem]{Problem}
\title{New Graphs  of Finite Mutation Type}
\author{Harm Derksen and Theodore Owen}
\thanks{This first author is partially supported by NSF grant DMS 0349019. 
This grant also supported the REU research project of the second author
on which this paper is based.}

\begin{document}
\begin{abstract}
To a directed graph without loops and $2$-cycles, we can associate a
skew-symmetric matrix with integer entries.
Mutations of such skew-symmetric matrices, and more generally skew-symmetrizable matrices,
have been defined in the context of cluster algebras by Fomin and Zelevinsky.
The mutation class of a graph $\Gamma$ is the set of all isomorphism classes
of graphs that can be obtained from $\Gamma$ by a sequence of mutations.
A graph is called mutation-finite if its mutation class is finite. 
Fomin, Shapiro and Thurston
constructed mutation-finite graphs from triangulations of oriented bordered surfaces
with marked points. We will call such graphs ``of geometric type''. Besides
graphs with $2$ vertices, and graphs of geometric type, there are only 9
other ``exceptional'' mutation classes that are known to be finite. In this paper
we introduce 2 new exceptional finite mutation classes.
\end{abstract}

\maketitle
Cluster algebras were introduced by Fomin and Zelevinsky in \cite{FZ1,FZ2} to create
an algebraic framework for total positivity nd canonical bases in semisimple algebraic groups.

An $n\times n$ matrix $B=(b_{i,j})$  is called {\em skew symmetrizable}
if there exists nonzero $d_1,d_2,\dots,d_n$ such that $d_ib_{i,j}=-d_jb_{j,i}$ for all $i,j$.
An {\em exchange matrix} is a skew-symmetrizable matrix with integer entries.

A {\em seed} is a pair $({\bf x}, B)$ where
$B$ is an echange matrix  and
${\bf x}=\{x_1,x_2,\dots,x_n\}$ is a set of $n$ algebraically independent elements.
For any $k$ with $1\leq k\leq n$ we define another seed
$({\bf x}',B')=\mu_k({\bf x},B)$ as follows. The matrix $B'=(b_{i,j}')$ is given by
$$
b_{i,j}'=\left\{\begin{array}{ll}
-b_{i,j}  & \mbox{if $i=k$ or $j=k$;}\\
b_{i,j}+[b_{i,k}]_+[b_{k,j}]_+-[-b_{i,k}]_+[-b_{k,j}]_{+} &
\mbox{otherwise.}\end{array}\right.
$$
Here, $[z]_+=\max\{z,0\}$ denotes the positive part of a real number $z$. Define
$$
{\bf x}'=\{x_1,x_2,\dots,x_{k-1},x_k',x_{k+1},\dots,x_n\}
$$
where $x_k'$ is given by
$$
x_k'=\frac{\prod_{i=1}^n x_i^{[b_{i,k}]_+}+\prod_{i=1}^n x_i^{[-b_{i,k}]_+}}{x_k}.
$$
Note that $\mu_k$ is an involution.
Starting with an initial seed $({\bf x},B)$ one can construct many seeds by
applying sequences of mutations. If $({\bf x}',B')$ is obtained from
the initial seed $({\bf x},B)$ by a sequence of mutations, then ${\bf x}'$ is called
a cluster, and the elements of ${\bf x}'$ are called {\em cluster variables}. 
The {\em cluster algebra} is the commutative subalgebra of $\Q(x_1,x_2,\dots,x_n)$ generated 
by all cluster variables.
A cluster algebra is called {\em of finite  type} if there are only 
finitely many seeds that can be obtained from the initial seed by sequences of
mutations. Cluster algebras of finite type were classfied in \cite{FZ2}.
\begin{example}[Cluster algebra of type ${\bf A}_1$]
If we start with the initial seed $({\bf x},B)$ where
${\bf x}=\{x_1,x_2\}$ and 
$$
B=\begin{pmatrix}
0 & -1\\
1 & 0
\end{pmatrix}
$$
Using mutations we get
$$
\{x_1,x_2\},\begin{pmatrix}
0 & 1\\
-1 & 0
\end{pmatrix}
\leftrightarrow
\{\frac{1+x_2}{x_1},x_2\},
\begin{pmatrix}
0 & -1\\
1 & 0
\end{pmatrix}
\leftrightarrow
\{\frac{1+x_2}{x_1},
\frac{1+x_1+x_2}{x_1,x_2}\},
\begin{pmatrix}
0 & 1\\
-1 & 0
\end{pmatrix}\leftrightarrow
$$
$$
\leftrightarrow
\{\frac{1+x_1}{x_2},
\frac{1+x_1+x_2}{x_1,x_2}\},
\begin{pmatrix}
0 & -1\\
1 & 0
\end{pmatrix}
\leftrightarrow
\{\frac{1+x_1}{x_2},
x_1\}
\begin{pmatrix}
0 & 1\\
-1 & 0
\end{pmatrix}
\leftrightarrow
\{x_2,x_1\},
\begin{pmatrix}
0 & -1\\
1 & 0
\end{pmatrix}
$$
The last seed 
$$
\{x_2,x_1\},
\begin{pmatrix}
0 & -1\\
1 & 0
\end{pmatrix}
$$
is considered the same as the initial seed. We just need
to exchange $x_1$ and $x_2$ (and accordingly swap the 2 rows
and swap the 2 columns in the exchange matrix) to get the initial seed.
\end{example}
A cluster algebra is called {\em mutation-finite} if only finitely many exchange matrices
appear in the seeds. Obviously a cluster algebra of finite type is mutation-finite.
But the converse is not true. For example, the exchange matrix
$$
B=\begin{pmatrix}
0 & -2\\
2 & 0
\end{pmatrix}
$$
gives a cluster algebra that is not of finite type. However, 
the only exchange matrix that apears is $B$ (and $-B$, but $-B$
is the same as $B$ after swapping the 2 rows and swapping  the  2 columns).

In this paper we will only consider exchange matrices that are already skew-symmetric.
To  a skew-symmetric $n\times n$ matrix $B=(b_{i,j})$ we can associate a directed graph $\Gamma(B)$ as follows. The vertices of the graph are labeled by $1,2,\dots,n$.
If $b_{i,j}>0$, draw $b_{i,j}$ arrows from $j$ to $i$. Any finite directed graph without
loops and 2 cycles can be obtained from a skew-symmetric exchange matrix in this way.
We can understand mutations
in terms of the graph. If $\Gamma=\Gamma(B)$ then $\mu_k\Gamma:=\Gamma(\mu_kB)$
is obtained from $\Gamma$ as follows. Start with $\Gamma$. For every incoming arrow $a:i\to k$ at $k$
and every outgoing arrow $b:k\to j$, draw a new composition arrow $ba:i\to j$.
Then, revert every arrow that starts or ends at $k$.
The graph now may have $2$-cycles. Discard $2$-cycles until there are now
more $2$-cycles left. The resulting graph is $\mu_k\Gamma$.
Two graphs are called {\em mutation-equivalent}, if one is obtained from the other
by a sequence of mutations and relabeling of the vertices.
The {\em mutation class} of a graph $\Gamma$ is the set of all isomorphism classes of graphs
that are mutation equivalent to $\Gamma$. A graph is mutation-finite if its mutation
class is finite.
\begin{convention}
In this paper, a subgraph of a directed graph $\Gamma$ will always mean a {\em full} subgraph,
i.e., for every two vertices $x,y$ in the subgraph, the subgraph also will contain
all arrows from $x$ to $y$.
\end{convention}

\section{Known mutation-finite connected graphs}
It is easy to see that a graph $\Gamma$ is mutation-finite if and only if each
of its connected components is mutation finite. We will discuss all known
examples of graphs of finite mutation type.
\subsection{Connected graphs with 2 vertices}
Let $\Theta(m)$ be the graph with two vertices $1,2$ and $m\geq 1$ arrows from $1$ to $2$.
The mutation class of $\Theta(m)$ is just the isomorphism class of $\Theta(m)$ itself.
So $\Theta(m)$ is mutation-finite. 
$$
\Theta(3):\xymatrix{\bullet \ar@<.5em>[r]\ar[r]\ar@<-.5em>[r] & \bullet}
$$
\subsection{Graphs from cluster algebras of finite type.}
An exchange matrix of a cluster algebra of finite type is mutation finite. The cluster
algebras of finite type were classified in \cite{FZ2}. This classification goes parallel
to the classification of simple Lie algebras, there are types
$${\bf A}_n,{\bf B}_n,{\bf C}_n,{\bf D}_n,{\bf E}_6,{\bf E}_7,{\bf E_8},{\bf F}_4,{\bf G}_2.$$
The types with a skew-symmetric exchange graph correspond
to the simply laced Dynkin diagrams  ${\bf A}_n,{\bf D}_n,{\bf E}_6,{\bf E}_7,{\bf E}_8$:
$$
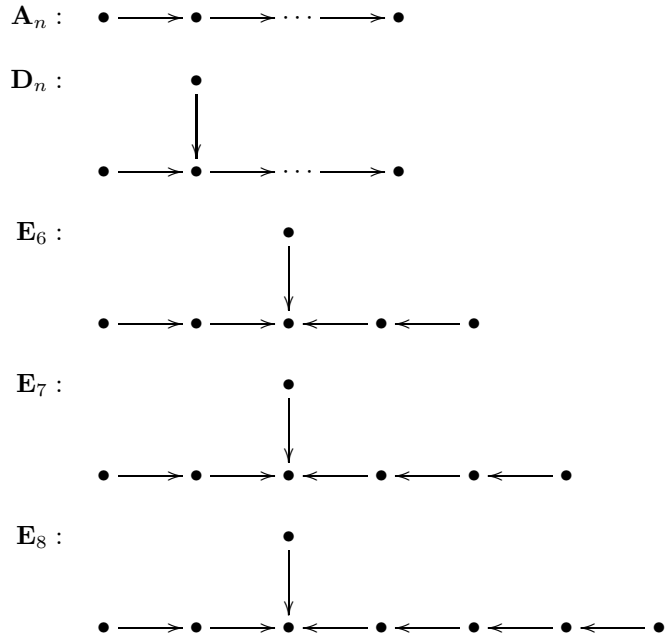
\begin{array}{rl}
{\bf A}_n: & 
\xymatrix{
\bullet \ar[r] & \bullet \ar[r] & \cdots\ar[r] & \bullet}\\ \\
{\bf D}_n: & 
\xymatrix{
& \bullet\ar[d] &  & &\\ 
\bullet \ar[r]& \bullet \ar[r] & \cdots \ar[r] & \bullet}\\ \\
{\bf E}_6: &
\xymatrix{
& &\bullet\ar[d] & &\\ 
\bullet\ar[r] & \bullet\ar[r] &\bullet & \bullet\ar[l] & \bullet\ar[l]}\\ \\
{\bf E}_7: &
\xymatrix{
& &\bullet\ar[d] & & &\\  
\bullet\ar[r] & \bullet\ar[r] &\bullet & \bullet\ar[l] & \bullet\ar[l] & \bullet \ar[l]}\\ \\
{\bf E}_8: &
\xymatrix{
& &\bullet\ar[d] & & & &\\ 
\bullet\ar[r] & \bullet\ar[r] &\bullet & \bullet\ar[l] & \bullet\ar[l] & \bullet \ar[l] & \bullet\ar[l]}
\end{array}
$$
The orientation of the arrows here were chosen somewhat arbitrarily. For each
diagram, a different choice of the orientation will still give the same mutation class.
\subsection{Graphs from extended Dynkin diagrams}
In \cite{AR} it was shown that a connected directed graph without oriented cycles is of finite mutation
type if and only if it has at 2 vertices (the graphs $\Theta(m)$, $m\geq 1$) or
the underlying undirected graph is an extended Dynkin diagrams. 
The type ${\bf D}$ and ${\bf E}$ extended Dynkin diagrams give
rise to the following finite mutation classes:
$$
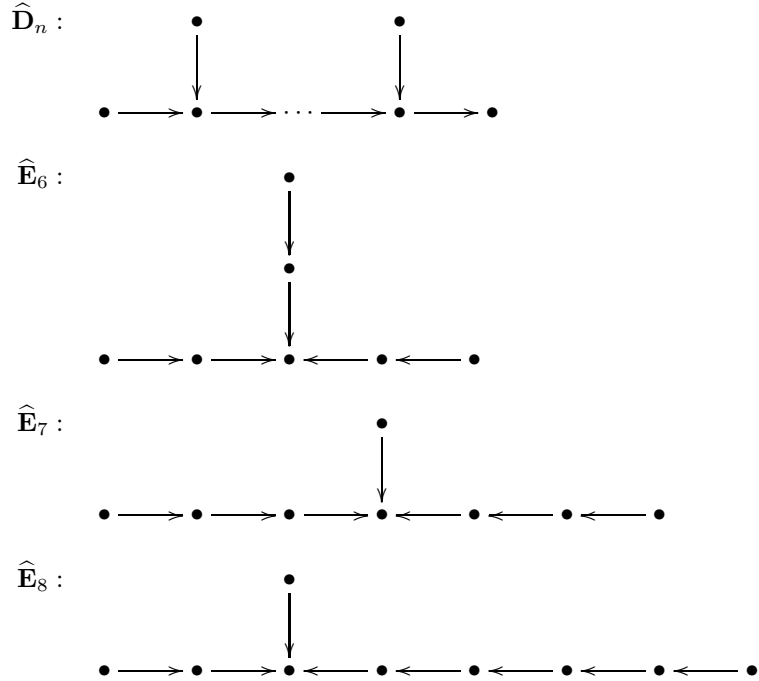
\begin{array}{rl}
\widehat{\bf D}_n: & 
\xymatrix{
& \bullet\ar[d] &  &\bullet\ar[d] & \\
\bullet \ar[r]& \bullet \ar[r] & \cdots \ar[r] & \bullet\ar[r] & \bullet}\\ \\
\widehat{\bf E}_6: &
\xymatrix{
& & \bullet\ar[d] & & \\
& &\bullet\ar[d] & &\\
\bullet\ar[r] & \bullet\ar[r] &\bullet & \bullet\ar[l] & \bullet\ar[l]}\\ \\
\widehat{\bf E}_7: &
\xymatrix{
& & &\bullet\ar[d] & & &\\
\bullet\ar[r] & \bullet\ar[r] & \bullet\ar[r] &\bullet & \bullet\ar[l] & \bullet\ar[l] & \bullet \ar[l]}\\ \\
\widehat{\bf E}_8: &
\xymatrix{
& &\bullet\ar[d] & & & & &\\
\bullet\ar[r] & \bullet\ar[r] &\bullet & \bullet\ar[l] & \bullet\ar[l] & \bullet \ar[l] & \bullet\ar[l] & \bullet\ar[l]}
\end{array}
$$
Again, for these types, a different choice for the orientations of the arrows still give
the same mutation class.
The diagram for $\widehat{\bf A}_{n}$ is an $(n+1)$-gon. If all arrows go clockwise 
or all arrows go counterclockwise, then we get the mutation class of ${\bf D}_n$.
Let $\widehat{\bf A}_{p,q}$ be the mutation class of the graph where $p$ arrows
go counterclockwise and $q$ arrows go clockwise, where $p\geq q\geq 1$.
For the mutation class it does  not matter which arrows are chosen
to be counterclockwise and with ones are chosen counterclockwise.
\subsection{Graphs coming from triangulations of surfaces}
In \cite{FST} the authors construct cluster algebras from bordered oriented surfaces
with marked points. These cluster algebras are always of finite mutation type.
The exchange matrices for these types are skew-symmetric.
The mutation-finite graphs that come from oriented bordered surfaces with
marked points will be called {\em of geometric type}.
In \S13 of that paper, the authors give a description of the graphs of geometric type.

A {\em block} is one of the diagrams below:
$$
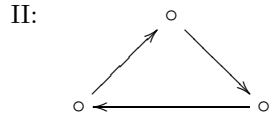
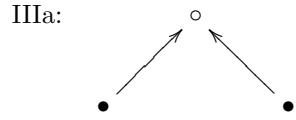
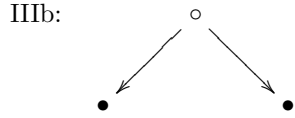
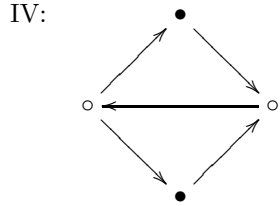
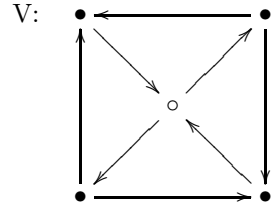
\begin{array}{rlrlrl}
\mbox{I:} &
\xymatrix{
\circ\ar[r] & \circ}
& 
\mbox{II:} &
\xymatrix{ & \circ\ar[rd] & \\
\circ\ar[ru] & & \circ\ar[ll]} &
\mbox{IIIa:} &
\xymatrix{ & \circ & \\
\bullet \ar[ru] & & \bullet\ar[lu]} \\ \\
\mbox{IIIb:} &
\xymatrix{ & \circ\ar[rd]\ar[ld] &\\
\bullet & & \bullet }  &
\mbox{IV:}
&
\xymatrix{ & \bullet\ar[rd] & \\
\circ\ar[ru]\ar[rd] & & \circ\ar[ll]\\
& \bullet\ar[ru] & }
&
\mbox{V:} &
\xymatrix{\bullet \ar[rd] & & \bullet \ar[dd]\ar[ll]\\
& \circ \ar[ru]\ar[ld] & \\
\bullet\ar[uu]\ar[rr] & & \bullet\ar[lu]}\end{array}
$$
Start with a disjoint union of blocks (every type may appear several times)
and choose a partial matching of the open vertices ($\circ$).
No vertex should be matched to  a vertex of the same block. 
Then construct a new graph by identifying the vertices that are matched to each other.
If in the resulting graph there are two vertices $x$ and $y$ with an arrow from $x$ to $y$
and an arrow from $y$ to $x$, then we omit both arrows (they cancel each other out).
A graph constructed in this way is called {\em block decomposable}. 
Fomin, Shapiro and Thurston prove in \cite[\S13]{FST}  that a graph is block decomposable if and only if the graph is of geometric type.

For example
$$
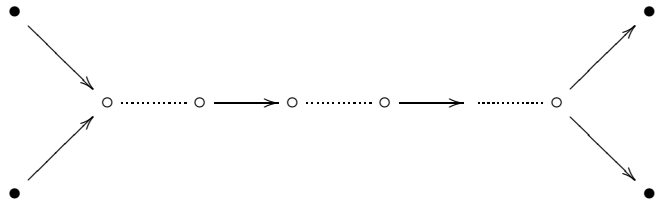
\xymatrix{
\bullet\ar[rd] & & & & & & &\bullet\\
& \circ\ar@{.}[r] &\circ\ar[r] & \circ\ar@{.}[r] & \circ\ar[r] & \ar@{.}[r] & \circ\ar[ru]\ar[rd] & \\
\bullet\ar[ru] &&&&&&& \bullet}
$$
gives
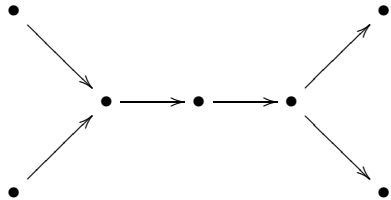
$$
\xymatrix{
\bullet\ar[rd] & & & & \bullet\\
& \bullet\ar[r] & \bullet \ar[r] & \bullet \ar[ru]\ar[rd] & \\
\bullet\ar[ru] & & & & \bullet}
$$
of type $\widehat{\bf D}_6$. It is easy to see that all graphs of type ${\bf A}_n,{\bf D}_n,
{\widehat {\bf A}}_{p,q},\widehat{\bf D}_n$ are block decomposable.
The partial matching
$$
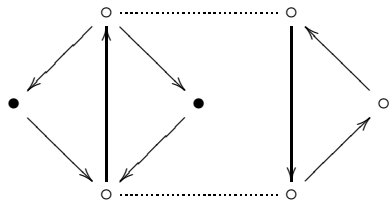
\xymatrix{
& \circ\ar[rd]\ar[ld]\ar@{.}[rr] & & \circ\ar[dd] &\\
\bullet\ar[rd] & & \bullet\ar[ld] & & \circ\ar[lu]\\
& \circ \ar[uu]\ar@{.}[rr] & & \circ\ar[ru] &}
$$
yields the block decomposable graph
$$
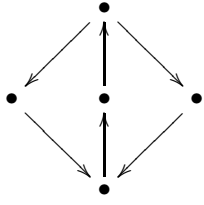
\xymatrix{ & \bullet\ar[rd]\ar[ld] & \\
\bullet \ar[rd] & \bullet \ar[u] & \bullet \ar[ld]\\
& \bullet\ar[u] & }.
$$
\subsection{Graphs of extended affine types}
The following graphs are also of finite mutation type:
$$
\begin{array}{rl}
{\bf E}_6^{(1,1)}: &
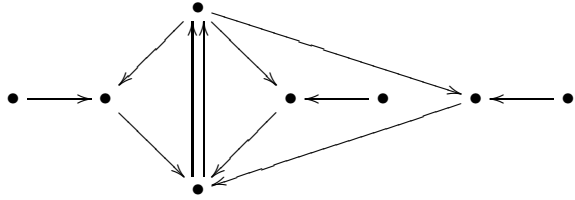
\xymatrix{
& & \bullet\ar[ld]\ar[rd]\ar[rrrd]  & & &  &\\ 
\bullet \ar[r] & \bullet\ar[rd] & &  \bullet\ar[ld] & \bullet\ar[l] & \bullet\ar[llld] & \bullet\ar[l]\\
& & \bullet\ar@<.5ex>[uu]\ar@<-.5ex>[uu] & &  &  &} \\
{\bf E}_7^{(1,1)}: &
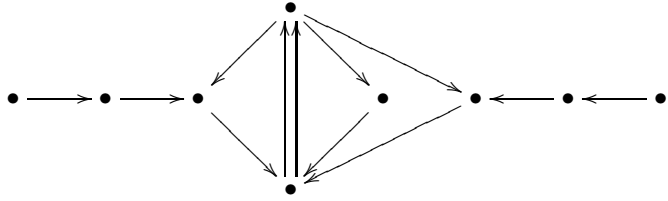
\xymatrix{
& & & \bullet\ar[ld]\ar[rd]\ar[rrd]   & & &  &\\ 
\bullet\ar[r] & \bullet \ar[r] & \bullet\ar[rd] & &  \bullet\ar[ld]  & \bullet\ar[lld] & \bullet\ar[l] & \bullet\ar[l]\\
& & & \bullet\ar@<.5ex>[uu]\ar@<-.5ex>[uu] &  & &  &}\\
{\bf E}_8^{(1,1)}: &
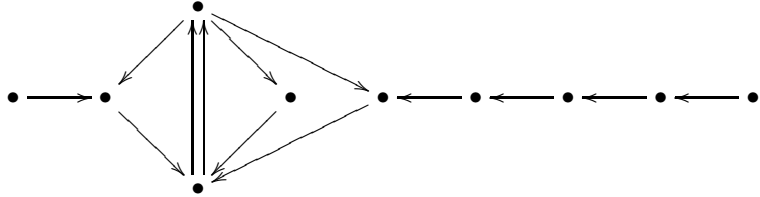
\xymatrix{
& & \bullet\ar[ld]\ar[rd]\ar[rrd]   & & &  & & &\\ 
 \bullet \ar[r] & \bullet\ar[rd] & &  \bullet\ar[ld]  & \bullet\ar[lld] & \bullet\ar[l] & \bullet\ar[l] & \bullet\ar[l] & \bullet \ar[l]\\
 & & \bullet\ar@<.5ex>[uu]\ar@<-.5ex>[uu] &  & & & &  &} 

 \end{array}
$$
These graphs are orientations  of the  Dynkin diagrams of  extended affine roots systems
first described by Saito  (see~ \cite[Table 1]{Saito}). The connection between extended
affine root systems and cluster combinatorics was first noticed by Geiss, Leclerc, and Schro\"er in
\cite{GLS}. It was shown using that these graphs are of finite mutation type by
an exhaustive computer search using the {\tt Java} applet for matrix mutations
written by Bernhard Keller (\cite{Keller}) and Lauren Williams.
\subsection{Summary}
All the known quivers of finite mutation type can be summerized as follows:
\begin{enumerate}
\item graphs of geometric type,
\item graphs with 2 vertices,
\item graphs in the 9 exceptional mutation classes
$$
{\bf E}_6,{\bf E}_7,{\bf E}_8,\widehat{\bf E}_6, \widehat{\bf E}_7,\widehat{\bf E}_8,
{\bf E}_6^{(1,1)},{\bf E}_7^{(1,1)},{\bf E}_8^{(1,1)}.
$$
\end{enumerate}
Fomin, Shapiro and Thurston asked to following question (see~\cite[Problem 12.10]{FST})\footnote{In the statement
of Problem 12.10 in \cite{FST}, the authors accidentally wrote $n\geq 2$ instead of $n\geq 3$. }.
\begin{problem}
Are these {\em all} connected graphs of finite mutation type? If not, are there only
finitely many exceptional finite mutation classes?
\end{problem}
In the next section, we will introduce 2 new mutation classes of finite type.

\section{New exceptional graphs of finite mutation-type}
\begin{proposition}
The following two graphs are of finite mutation type:
$$
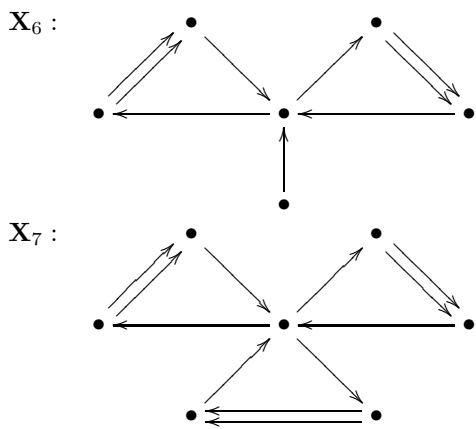
\begin{array}{rl}
{\bf X}_6: & 
\xymatrix{
& \bullet\ar[rd] & & \bullet\ar@<.5ex>[rd]\ar@<-.5ex>[rd] \\
\bullet\ar@<.5ex>[ru]\ar@<-.5ex>[ru] & & \bullet\ar[ll]\ar[ru] & & \bullet\ar[ll]\\
& &\bullet \ar[u]& & }\\
{\bf X}_7: & 
\xymatrix{
& \bullet\ar[rd] & & \bullet\ar@<.5ex>[rd]\ar@<-.5ex>[rd] \\
\bullet\ar@<.5ex>[ru]\ar@<-.5ex>[ru] & & \bullet\ar[ll]\ar[ru]\ar[rd] & & \bullet\ar[ll]\\
& \bullet\ar[ru] && \bullet\ar@<.5ex>[ll]\ar@<-.5ex>[ll] & }\end{array}
$$
\end{proposition}
\begin{proof}
This is easy to verify by hand or by using the applet~\cite{Keller}.
The mutation classes for ${\bf X}_6$ and ${\bf X}_7$ are suprisingly small. The
mutation class of $X_6$ consists of the following $5$ graphs:
$$
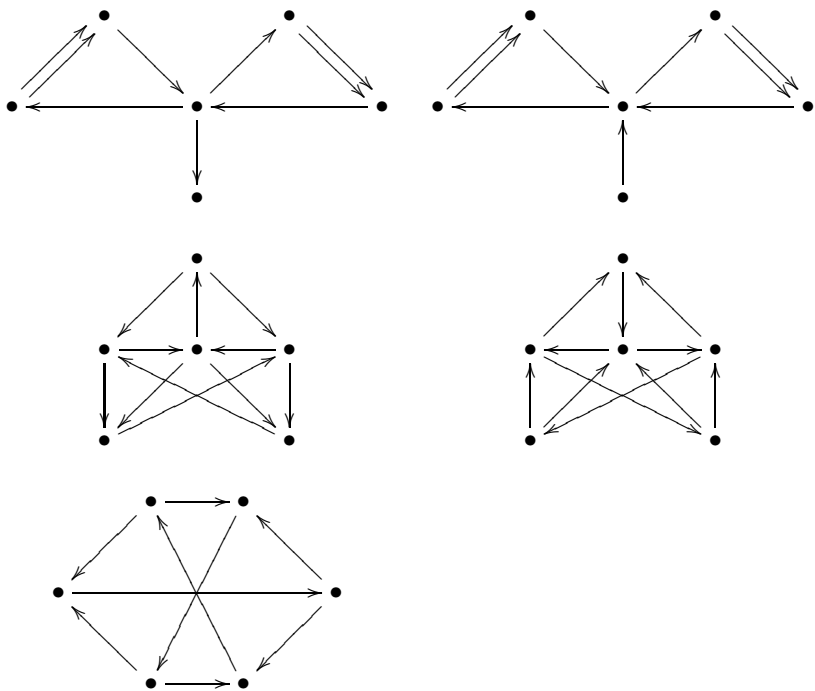
\begin{array}{cc}
\xymatrix{
& \bullet\ar[rd] & & \bullet\ar@<.5ex>[rd]\ar@<-.5ex>[rd] \\
\bullet\ar@<.5ex>[ru]\ar@<-.5ex>[ru] & & \bullet\ar[ll]\ar[ru]\ar[d] & & \bullet\ar[ll]\\
& &\bullet & & } &
\xymatrix{
& \bullet\ar[rd] & & \bullet\ar@<.5ex>[rd]\ar@<-.5ex>[rd] \\
\bullet\ar@<.5ex>[ru]\ar@<-.5ex>[ru] & & \bullet\ar[ll]\ar[ru] & & \bullet\ar[ll]\\
& &\bullet\ar[u] & & } \\ \\
\xymatrix{
& \bullet\ar[ld]\ar[rd] & \\
\bullet \ar[r]\ar[d] & \bullet\ar[u]\ar[ld]\ar[rd] & \bullet\ar[l]\ar[d]\\
\bullet\ar[rru] & & \bullet\ar[llu]} &
\xymatrix{
& \bullet\ar[d] & \\
\bullet\ar[ru]\ar[rrd] & \bullet\ar[l]\ar[r] & \bullet\ar[lu]\ar[lld] \\
\bullet\ar[u]\ar[ru] & & \bullet \ar[lu]\ar[u]} \\ \\
\xymatrix{
& \bullet \ar[r]\ar[ld] & \bullet\ar[ldd] & \\
\bullet\ar[rrr] & & & \bullet \ar[lu]\ar[ld]\\
& \bullet\ar[lu]\ar[r] & \bullet\ar[luu] & } &
\end{array}
$$
The mutation class of ${\bf X}_7$ consists of the following 2 graphs:
$$
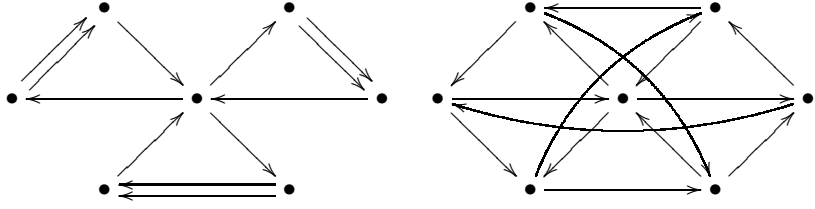
\begin{array}{cc}
\xymatrix{
& \bullet\ar[rd] & & \bullet\ar@<.5ex>[rd]\ar@<-.5ex>[rd] \\
\bullet\ar@<.5ex>[ru]\ar@<-.5ex>[ru] & & \bullet\ar[ll]\ar[ru]\ar[rd] & & \bullet\ar[ll]\\
& \bullet\ar[ru] && \bullet\ar@<.5ex>[ll]\ar@<-.5ex>[ll] & } &
\xymatrix{
& \bullet\ar@/^1pc/[rrdd]\ar[ld] &&  \bullet\ar[ll]\ar[ld] & \\
 \bullet\ar[rd]\ar[rr] & & \bullet \ar[lu]\ar[rr]\ar[ld] & & \bullet \ar[lu]\ar@/^1pc/[llll]\\
& \bullet\ar@/^1pc/[rruu]\ar[rr] & & \bullet\ar[lu]\ar[ru]}
\end{array}
$$
\end{proof}
\begin{corollary}
The graphs ${\bf X}_6$ and ${\bf X}_7$ are not mutation-equivalent
to 
$$
{\bf E}_6,{\bf E}_7,{\bf E}_8,\widehat{\bf E}_6, \widehat{\bf E}_7,\widehat{\bf E}_8,
{\bf E}_6^{(1,1)},{\bf E}_7^{(1,1)},{\bf E}_8^{(1,1)}.$$
\end{corollary}
\begin{proof}
The reader easily verifies that none of these $9$ graphs are in the mutation
classes of ${\bf X}_6$ or ${\bf X}_7$.
\end{proof}
\begin{proposition}
The graphs ${\bf X}_{6}$ and ${\bf X}_7$ are not block decomposable. In particular,
these graphs do not come from oriented surfaces with marked points.
\end{proposition}
\begin{proof}
Suppose that ${\bf X}_6$ is block decomposable. None of the blocks
will be of type V, since the block V contains a $4$-cycle
which will not vanish after the matching, and ${\bf X}_6$ does not contain a $4$-cycle.
We label the vertices of ${\bf X}_6$ as follows:
\begin{equation}\label{label}
\xymatrix{
& z_1 \ar[rd] & & y_2 \ar@<.5ex>[rd]\ar@<-.5ex>[rd] \\
y_1 \ar@<.5ex>[ru]\ar@<-.5ex>[ru] & & x \ar[ll]\ar[ru] & & z_2 \ar[ll]\\
& & w \ar[u] & & }
\end{equation}
At vertex $x$ there are $2$ arrows going out and $3$ arrows coming in.
To form a graph with this property from the blocks, we must either glue
blocks II and IV along $x$, or glue blocks IIIa and IV along $x$.

If we glue blocks IIIa and IV we get the following graph:
$$
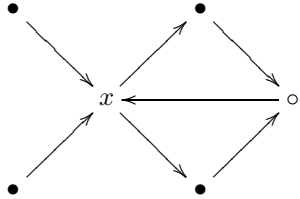
\xymatrix{
\bullet\ar[rd] & & \bullet \ar[rd] & \\
& x  \ar[ru]\ar[rd] & & \ar[ll] \circ\\
\bullet\ar[ru] & & \bullet \ar[ru] & }
$$
where the open vertex ($\circ$) may be matched further with other blocks. Even
after further matching, $x$ will have at least 2 neighbors which are only connected to $x$.

If blocks II and IV are matched to form a vertex $x$, then we get the following graph
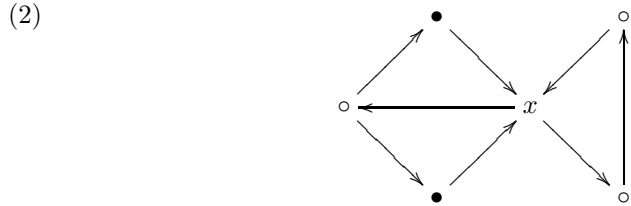
\begin{equation}\label{eq1}
\xymatrix{
& \bullet \ar[rd] & & \circ\ar[ld]\\
\circ\ar[ru]\ar[rd] & & x\ar[ll]\ar[rd] & \\
& \bullet\ar[ru] & & \circ\ar[uu]}
\end{equation}
Here the open vertices can be matched further. However, they cannot be matched
among themselves, because this would change the number of incoming
and outgoing arrows at $x$. 
In ${\bf X}_6$, $x$ has incoming arrows from $w,z_1,z_2$.
So in (\ref{eq1}), one of the vertices marked with $\bullet$ has to
correspond to $z_1$ or $z_2$. But it is clear that even after further
matching, the vertices marked with $\bullet$ will only have $1$ incoming arrow.
Contradiction.
This shows that ${\bf X}_6$ is not block decomposable. 
Therefore, ${\bf X}_6$ does not come from an a triangulation of
an oriented surface with marked points.

Since ${\bf X}_6$ is a subgraph of ${\bf X}_7$, ${\bf X}_7$ does not come
from a triangulation of an oriented surface with marked points either.
\end{proof}

\section{Mutation-finite graphs containing ${\bf X}_6$ or ${\bf X}_7$}

The following result was proven in~\cite{ABBS}. We include the short
proof for the reader's convenience.
\begin{theorem}\label{theo4}
The finite mutation classes of connected quivers with 3 vertices are:
$$
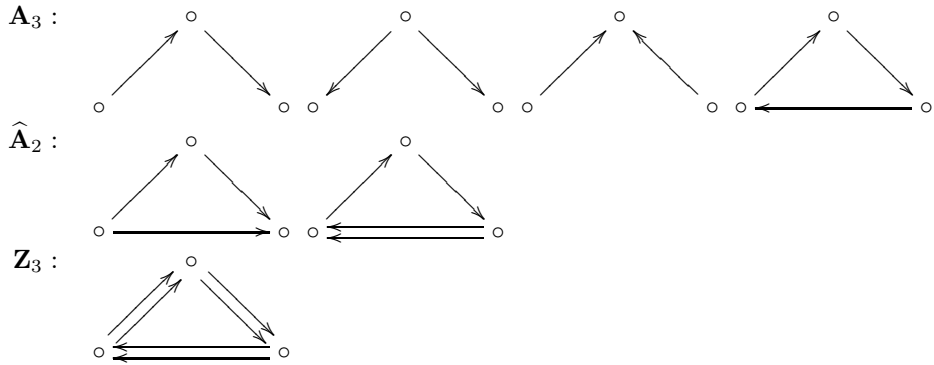
\begin{array}{rl}
{\bf A}_3: & 
\xymatrix{
& \circ \ar[rd] & \\
\circ\ar[ru] & & \circ}
\xymatrix{
& \circ\ar[rd]\ar[ld] &\\
\circ & & \circ}
\xymatrix{
& \circ & \\
\circ \ar[ru] & & \circ\ar[lu]}
\xymatrix{
& \circ \ar[rd] &\\
\circ\ar[ru] & & \circ\ar[ll]}\\
\widehat{\bf A}_2: &
\xymatrix{
& \circ\ar[rd] & \\
\circ\ar[ru]\ar[rr] & & \circ}
\xymatrix{
& \circ\ar[rd] &\\
\circ \ar[ru] & & \circ\ar@<.5ex>[ll]\ar@<-.5ex>[ll]}\\
{\bf Z}_3: & 
\xymatrix{
& \circ\ar@<.5ex>[rd]\ar@<-.5ex>[rd] & \\
\circ\ar@<.5ex>[ru]\ar@<-.5ex>[ru] & & \circ\ar@<.5ex>[ll]\ar@<-.5ex>[ll]}
\end{array}
$$

\end{theorem}
\begin{proof}
Suppose that $\Gamma$ is a connected graph of finite mutation type with 3 vertices.
Assume that $\Gamma$, among all the graphs in its mutation class,
has the largest  possible number of arrows.
Without loss of generality we may assume that
$\Gamma$ is of the form
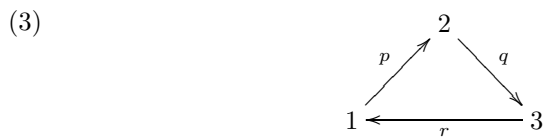
\begin{equation}\label{form1}
\xymatrix{
& 2\ar[rd]^{q} &&\\
1\ar[ru]^p & & 3\ar[ll]^{r}}
\end{equation}
where $p,q,r\geq 0$ denote the number of arrows.
If $\Gamma$ is not of the form (\ref{form1}), then it is of the form
\begin{equation}\label{form2}
\xymatrix{
& 2\ar[rd]^{q} & \\
1\ar[ru]^p\ar[rr]_r & & 3}.
\end{equation}
and mutation at vertex $2$ will not decrease the number of
arrows, and we obtain a graph of the form (\ref{form1}).
Without loss of generality we may assume that 
\begin{equation}\label{eq11}
p,q\geq r.
\end{equation}
In particular, 
\begin{equation}\label{eq115}
p,q\geq 1
\end{equation}
otherwise the graph would not be connected.
After mutation at vertex $2$, we get 
\begin{equation}\label{eq12}
\xymatrix{
& 2\ar[ld]_{p} &\\
1\ar[rr]_{pq-r} & & 3\ar[lu]_{q}}
\end{equation}
Since $\Gamma$ had the maximal number of arrows, we have $pq-r\leq r$, so 
\begin{equation}\label{eq13}
pq\leq 2r.
\end{equation}
From $r^2\leq pq\leq 2r$ follows that $r=0,1,2$.
If $r=0$, then $pq=0$ by (\ref{eq13}) which contradicts (\ref{eq115}).
If $r=1$, then $pq=1$ or $pq=2$ by (\ref{eq13}).
If $pq=1$, then $p=q=1$ wich yields type ${\bf A}_3$.
If $pq=2$ then $p=2$, $q=1$ or $p=1$, $q=2$. In either case
we get type $\widehat{\bf A}_2$.
If $r=2$, then (\ref{eq11}) and (\ref{eq13}) imply that $p=q=2$
and we obtain type ${\bf Z}_3$.
\end{proof}
\begin{corollary}\label{coratmost2}
If $\Gamma$ is a graph of finite mutation type with $\geq 3$ vertices
Then the number of arrows between any $2$ vertices is at most $2$.
\end{corollary}
\begin{proof}
Suppose $x$ and $y$ are vertices of $\Gamma$ with $p\geq 1$ arrows
from $x$ to $y$. Since $\Gamma$ is connected, there exists a vertex $z$
that is connected to $x$ or $y$.The subgraph with vertices $x,y,z$ is also
of finite mutation type. From the classification in Theorem~\ref{theo4} it is clear that
$p\leq 2$.
\end{proof}
\begin{definition}
An {\em obstructive sequence} for a graph $\Gamma$ is a sequence
of vertices $x_1,x_2,x_3,\dots,x_{\ell}$ such that the mutated graph
$$
\mu_{x_\ell}\cdots \mu_{x_2}\mu_{x_1}\Gamma
$$
has two vertices with at least $3$ arrows between them.
\end{definition}
By Corollary~\ref{coratmost2},  a graph for which an obstructive
sequence exists cannot be of finite mutation type.

\begin{lemma}\label{lem6}
If $\Gamma$ is a mutation-finite connected graph with $\geq 4$ vertices,
then $\Gamma$ cannot contain ${\bf Z}_3$ as a subgraph.
\end{lemma}
\begin{proof}
Suppose $\Gamma$ is a mutation-finite connected graph with $\geq 4$ vertices
containing ${\bf Z}_3$. Then $\Gamma$ has a mutation-finite connected
subgraph with exactly $4$ vertices containing ${\bf Z}_3$.
Without loss of generality we may assume that $\Gamma$ has $4$ vertices.
We label the vertices of ${\bf Z}_3$ as follows
$$
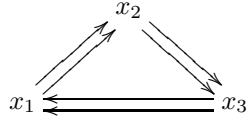
\xymatrix{
& x_2\ar@<.5ex>[rd]\ar@<-.5ex>[rd] & \\
x_1\ar@<.5ex>[ru]\ar@<-.5ex>[ru] & & x_3\ar@<.5ex>[ll]\ar@<-.5ex>[ll]}
$$
Let $y$ be the $4$-th vertex of $\Gamma$.
Now $y$ is connected to at least one of the vertices of ${\bf Z}_3$, say to $x_1$.
Because the subgraph with vertices $\{y,x_1,x_2\}$ is of finite mutation type, 
all arrows must go from $y$ to $x_1$ and not in the opposite direction by Theorem~\ref{theo4}.
But because the subgraph with vertices $\{y,x_1,x_3\}$ is of finite mutation type, 
all arrows must go from $x_1$ to $y$. Contradiction.
\end{proof}
\begin{corollary}\label{cor8}
If $\Gamma$ is a mutation-finite graph with $\geq 4$ vertices, then every
connected subgraph with $3$ vertices must be of type ${\bf A}_3$ or $\widehat{\bf A}_2$.
\end{corollary}

\begin{theorem}\label{theo9}
The only connected mutation-finite quiver with 7 vertices containing ${\bf X}_6$ is ${\bf X}_7$. 
\end{theorem}
\begin{proof}
Suppose that $\Gamma$ is a mutation finite quiver containin ${\bf X}_6$.
We label the vertices of ${\bf X}_{6}$ as shown in (\ref{label}),
and denote the other vertex by $u$.
Since $\Gamma$ is connected, $u$ is connected to at least 1 other vertex.

{\bf case 1:} Suppose that $u$ is connected to 1 vertex of ${\bf X}_6$. 
If $u$ is connected to $y_1$ or $z_1$ then the subgraph with vertices
$\{u,y_1,z_1\}$ is not of type ${\bf A}_3$ or $\widehat{A}_2$, contradicting Corollary~\ref{cor8}.
Similarly $u$ is not connected to $y_2$ or $z_2$. 

Suppose $u$ is connected to $x$. After mutation at $u$ we may assume that
arrows go from $u$ to $x$. From Corollary~\ref{cor8} applied to  the subgraph
with vertices $\{u,x,w\}$ follows that there can be only 1 arrow.

Now $w,x,y_1,z_1,w,u$ is a obstructive sequence.

Suppose that $u$ is connected to $w$. Without loss of generality we may assume
that arrows go from $u$ to $w$. By Corollary~\ref{cor8} applied to the subgraph
with vertices $\{u,w,x\}$ there is at most 1 arrow.
Now $x,w,y_1,y_2,z_1,z_2,x,w,y_2,$ is a obstructive sequence.

attached to $X_{6}$ by only one set of arrows.  

{\bf Case 2:}
Suppose that $u$ is connected to 2 vertices. If $u$ is connected to $y_1$ or $z_1$,
then the only possibility is that there is one arrow from $u_1$ to $y_1$ and one
arrow from $z_1$ to $u_1$. An obstructive sequence is $u,x,y_1,y_2.x,w,x,z_2,y_1$.
Similarly, $u$ cannot be connected to $y_2$ or $z_2$.

Therefore, $u$ must be connected to $w$ and $x$. The only cases that avoid
a connected subgraph with 3 vertices not of type ${\bf A}_3$ or $\widehat{\bf A}_2$ are:
$$
{\bf (a)}:
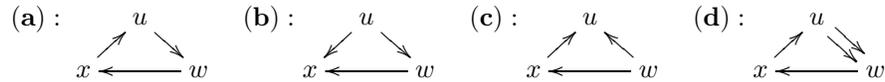
\xymatrix@-1.2pc{
& u\ar[rd] & \\
x\ar[ru] & & w\ar[ll]
}\quad
{\bf (b):}
\xymatrix@-1.2pc{
& u\ar[ld]\ar[rd]& \\
x & & w\ar[ll]
}\quad
{\bf (c):}
\xymatrix@-1.2pc{
& u &\\
x\ar[ru] & & w\ar[ll]\ar[lu]}
\quad {\bf (d):}
\xymatrix@-1.2pc{
& u \ar@<.5ex>[rd]\ar@<-.5ex>[rd] & \\
x\ar[ru] & & w\ar[ll]}
$$
Case (a) reduces to Case 1 after mutation at vertex $u$.
Cases (b) and (c) give isomorphic graphs.
Case (b) has an obstructive sequence $w,x,u$.
Case (d) gives us the graph ${\bf X}_7$.

\emph{Case 3: 3 attachments}

Vertex u must be attached to either vertices $y_1$ and $z_1$ or vertices $y_2$ and $z_2$. 
Without loss of generality, we may assume that $u$ is connected to both $y_1$ and $z_1$.
There is one arrow from $u$ to $y_1$ and one arrow from $z_1$ to $u$.
The vertex $u$ is also connected to either $w$ or $x$.
There are $4$ cases:
$$
{\bf (a)}:
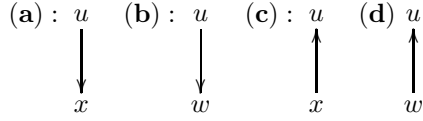
\xymatrix{
u\ar[d] \\ x}\quad
{\bf (b)}:
\xymatrix{u \ar[d] \\ w} \quad
{\bf (c)}:
\xymatrix{ u \\ x\ar[u]}\quad
{\bf (d)}
\xymatrix{ u \\ w\ar[u]}
$$
Case (a) has the obstructive sequence $u,x,z_1$. Case (b) has the obstructive sequence $u,x,z_1,w,x$.
Case (c) has the obstructive sequence $x,w,u,y_1$. Case (d) has the obstructive sequence $u,x,z_2,y_1,w$.

{\bf Case 4:}
 Suppose that $u$ is connected to $4$ vertices.
 If $u$ is connected to $y_1,z_1,y_2,z_2$ then there is an arrow from
 $u$ to $y_1$ and to $y_2$ and arrows from $z_1$ and $z_2$ to $u$.
 An obstructive sequence is $u,x,w,y_2,z_1$.
  
 Otherwise, $u$ must be connected to either $y_1$ and $z_1$, or to $y_2$ and $z_2$,
 but not both. Without loss of generality we may assume that
 $u$ is connected to $y_1$ and $z_1$.
 Now $u$ is also connected to $w$ and $x$.
 The only cases that avoid
a connected subgraph with 3 vertices not of type ${\bf A}_3$ or $\widehat{\bf A}_2$ are:
$$
{\bf (a)}:
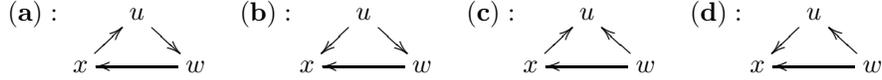
\xymatrix@-1.2pc{
& u\ar[rd] & \\
x\ar[ru] & & w\ar[ll]
}\quad
{\bf (b):}
\xymatrix@-1.2pc{
& u\ar[ld]\ar[rd]& \\
x & & w\ar[ll]
}\quad
{\bf (c):}
\xymatrix@-1.2pc{
& u &\\
x\ar[ru] & & w\ar[ll]\ar[lu]}
\quad {\bf (d):}
\xymatrix@-1.2pc{
& u \ar[ld] & \\
x & & w\ar[ll]\ar[lu]}
$$
Case (a) has the obstructive sequence $u,x,y_1$. Case (b) has the obstructive sequence $u,x,z_1$.
Case (c) has the obstructive sequence $u,x$. And case (d) is mutation equivalent to case (c)
via the mutation at $w$.

{\bf Case 5}: 
Suppose that $u$ is connected to $5$ of the vertices of ${\bf X}_6$.
Then $u$ must be connected to $y_1,z_1,y_2,z_2$, with arrows
going from $u$ to $y_1$ and $y_2$ and arrows going from $z_1$ and $z_2$ to $u$.
Now $u$ must be connected to either $x$ or $w$. There are $4$ subcases:
$$
{\bf (a)}:
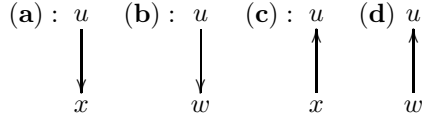
\xymatrix{
u\ar[d] \\ x}\quad
{\bf (b)}:
\xymatrix{u \ar[d] \\ w} \quad
{\bf (c)}:
\xymatrix{ u \\ x\ar[u]}\quad
{\bf (d)}
\xymatrix{ u \\ w\ar[u]}
$$
Case (a) has the  obstructive sequence $u,x$. Case (b) has the obstructive sequence $u,x,w$.
Case (c) has the  obstructive sequence $u,x$. Case (d) has the obstructive sequence $u,x,y_1,y_2$.

{\bf Case 6:}
The vertex $u$ is connected to all $6$ vertices of ${\bf X}_6$.
There must be arrows from $u$ to $y_1$ and $y_2$ and arrows from $z_1$ and $z_2$ to $u$.
The only possibilities to connect $u$ to $x$ and $w$ that avoid
a connected subgraph with 3 vertices not of type ${\bf A}_3$ or $\widehat{\bf A}_2$ are:
$$
{\bf (a)}:
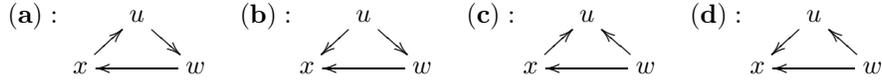
\xymatrix@-1.2pc{
& u\ar[rd] & \\
x\ar[ru] & & w\ar[ll]
}\quad
{\bf (b):}
\xymatrix@-1.2pc{
& u\ar[ld]\ar[rd]& \\
x & & w\ar[ll]
}\quad
{\bf (c):}
\xymatrix@-1.2pc{
& u &\\
x\ar[ru] & & w\ar[ll]\ar[lu]}
\quad {\bf (d):}
\xymatrix@-1.2pc{
& u \ar[ld] & \\
x & & w\ar[ll]\ar[lu]}
$$
Case (a) has the obstructive sequence $x,y_1,y_2$. Case (b) has the obstructive sequence
$x,z_1,z_2$. Case 3 has  the obstructive sequence $x,y_1,y_2$. Case 4 has the  obstructive sequence
$x,z_1,z_2$.

Therefore the only connected mutation-finite quiver with 7 vertices containing $X_{6}$ is $X_{7}$.
\end{proof}
\begin{theorem}\label{theo10}
There is no connected mutation-finite quiver with $\geq 8$ vertices containing ${\bf X}_{7}$.
\end{theorem}
\begin{proof}
Suppose that $\Gamma$ is a connected graph with $\geq 8$ vertices containing ${\bf X}_7$.
We will show that this will lead to a contradiction.
The graph  $\Gamma$ contains a connected subgraph with exactly $8$ vertices
which contains ${\bf X}_7$. So without loss of generality we may assume that $\Gamma$
has exactly $8$ vertices.
Let us label the vertices of ${\bf X}_7$ as follows:
$$
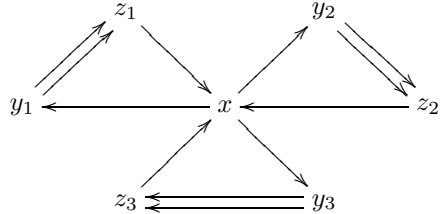
\xymatrix{
& z_1\ar[rd] & & y_2\ar@<.5ex>[rd]\ar@<-.5ex>[rd] \\
y_1 \ar@<.5ex>[ru]\ar@<-.5ex>[ru] & & x \ar[ll]\ar[ru]\ar[rd] & & z_2\ar[ll]\\
& z_3\ar[ru] & & y_3\ar@<.5ex>[ll]\ar@<-.5ex>[ll] & }
$$
Suppose that $u$ is the $8$-th vertex of $\Gamma$.
Because of the symmetry, we may assume, without loss of generality that $
u$ is connected to $x$, $y_1$ or $z_1$.
The subgraph with vertices $\{x,y_1,z_1,y_2,z_2,z_3,u\}$ is connected
and  contains ${\bf X}_7$. By Theorem~\ref{theo9} this graph must be isomorphic
to ${\bf X}_7$. This means that there must be $2$ arrows from $u$ to $z_3$.
Now the subgraph with vertices  $y_3,z_3,u$ cannot be of finite mutation type
because of Theorem~\ref{theo4}.
\end{proof}
\begin{corollary}
No graph of one of the 9 exceptional types contains
$X_{6}$ or $X_{7}$ as a subgraph.
\end{corollary}

\section{Conclusion}

We have exhibited two graphs which are of finite mutation type, but which are not of geometric type or 
 isomorphic to the 9 known exceptional cases. It is natural to ask whether there are any more
 exceptional graphs of finite mutation type.
 \begin{problem}
 Is it true that every  finite mutation class of graphs with $\geq 3$ vertices
 which is not of geometric type is one of the following 11 mutation classes:
 $$
 {\bf E}_6,{\bf E}_7, {\bf E}_8, \widehat{\bf E}_6,\widehat{\bf E}_7,\widehat{\bf E}_8, 
 {\bf E}_6^{(1,1)},{\bf E}_7^{(1,1)},{\bf E}_8^{(1,1)}, {\bf X}_6,{\bf X}_7?
 $$
 \end{problem}
\noindent{\bf Acknowledgements.}
The authors would like to thank Ralf Spatzier for bringing them together in this REU project. 
The second author
likes to thank Tracy Payne for getting him interested in the REU program,
and Christine Betz Bolang for help getting into the REU program.


\begin{thebibliography}{20}
\bibitem{ABBS} I.~Assem, M.~Blais, Th.~Br\"ustle, A.~Samson, {\it Mutation classes of
skew-symmetric $3\times 3$ matrices}, to appear in Comm. Alg.
\bibitem{AR} A.~B.~Buan, I.~Reiten, {\it Acyclic quivers of finite mutation type}, International Math. Research Notices {\bf 2006}, Art. ID 12804.
\bibitem{BFZ3} A.~Berenstein,  S.~Fomin, A.~Zelevinsky, {\it Cluster algebras. III. Upper bounds and double Bruhat cells}, Duke Math. J.~{\bf 126} (2005), no.~1, 1--52.
\bibitem{FST}
S.~Fomin, M.~Shapiro, D.~Thurston, {\it cluster algebras and triangulated surfaces part I: cluster complexes}, to appear in Acta Math.
\bibitem{FZ1}
S.~Fomin, A.~Zelevinsky, {\it Cluster algebras I: Foundations}, J. Amer. Math. Soc.~{\bf 15} (2002), 497--529.
\bibitem{FZ2}S.~Fomin, A.~Zelevinsky, {\it Cluster algebras II: Finite type classification},
Invent.~Math.~{\bf 154} (2003), 63--121.
\bibitem{FZ4}S.~Fomin, A.~Zelevinsky, {\it Cluster algebras IV: Coefficients}, Compos. Math.~{\bf 143} (2007), no.~2, 112--164.
\bibitem{GLS} C.~Geiss, B.~Leclerc, J.~Schr\"oer, {\it Semicanonical bases and preprojective algebras},
Ann. Sci.\ \'Ecole Norm. Sup.~(4)~{\bf 38} (2005), no.~2, 193--253.
\bibitem{Keller}B.~Keller, Quiver mutation in Java, {\tt http://www.math.jussieu.fr/\~{}keller/quivermutation/}
\bibitem{Saito}K.~Saito, {\it Extended affine root systems. I. Coxeter transformations}, Publ. Res. Inst. Math. Sci.~{\bf 21} (1985), no.~1, 75--179.
\end{thebibliography}
\end{document}